\documentclass[11pt]{amsart}
\usepackage{amssymb}
\usepackage{multirow}
\usepackage{mathtools}
\usepackage[pagebackref]{hyperref}
\hypersetup{breaklinks=true}
\usepackage{breakurl}
\usepackage{enumitem}
\usepackage{tikz-cd}

\usepackage{color}

\newcommand{\A}{\mathbb{A}}
\newcommand{\B}{\mathcal{B}}

\newcommand{\C}{\mathbb{C}}
\newcommand{\CV}{\mathcal{CV}}
\newcommand{\cZ}{\mathcal{Z}}
\newcommand{\F}{\mathbb{F}}
\newcommand{\m}{\mathfrak{m}}
\newcommand{\M}{\mathcal{M}}
\newcommand{\OK}{\mathcal{O}}
\newcommand{\p}{\mathfrak{p}}

\newcommand{\Q}{\mathbb{Q}}
\newcommand{\V}{\mathcal{V}}
\newcommand{\Z}{\mathbb{Z}}

\DeclareMathOperator{\ab}{ab}

\DeclareMathOperator{\Aut}{Aut}

\DeclareMathOperator{\dR}{dR}
\DeclareMathOperator{\End}{End}
\DeclareMathOperator{\et}{et}
\DeclareMathOperator{\Frob}{Frob}
\DeclareMathOperator{\FVS}{FVS}
\DeclareMathOperator{\Gal}{Gal}
\DeclareMathOperator{\GL}{GL}
\DeclareMathOperator{\Hom}{Hom}
\DeclareMathOperator{\HS}{HS}

\DeclareMathOperator{\Jac}{Jac}
\DeclareMathOperator{\num}{num}
\DeclareMathOperator{\ord}{ord}
\DeclareMathOperator{\rat}{rat}
\DeclareMathOperator{\mhom}{hom}
\DeclareMathOperator{\Rep}{Rep}
\DeclareMathOperator{\Res}{Res}

\usepackage[margin=1in]{geometry}

\newtheorem{thm}{Theorem}[section]
\newtheorem{cor}[thm]{Corollary}
\newtheorem{prop}[thm]{Proposition}
\newtheorem{lem}[thm]{Lemma}

\theoremstyle{definition}
\newtheorem{defn}[thm]{Definition}

\newtheorem{question}[thm]{Question}

\theoremstyle{remark}
\newtheorem{rem}[thm]{Remark}

\begin{document}

\title{Chow motives associated to certain algebraic Hecke characters}
\author{Laure Flapan}
\address{Department of Mathematics, Northeastern University, USA}
\email{l.flapan@northeastern.edu}

\author{Jaclyn Lang}
\address{Max Planck Institute for Mathematics, Germany}
\email{jlang@mpim-bonn.mpg.de}


\begin{abstract}
Shimura and Taniyama proved that if $A$ is a potentially CM abelian variety over a number field $F$ with CM by a field $K$ linearly disjoint from F, then there is an algebraic Hecke character $\lambda_A$ of $FK$ such that $L(A/F,s)=L(\lambda_A,s)$.  We consider a certain converse to their result.  Namely, let $A$ be a potentially CM abelian variety appearing as a factor of the Jacobian of a curve of the form $y^e=\gamma x^f+\delta$. Fix positive integers $a$ and $n$ such that $n/2 < a \leq n$.  Under mild conditions on $e, f, \gamma, \delta$, we construct a Chow motive $M$, defined over $F=\Q(\gamma,\delta)$, such that $L(M/F,s)$ and $L(\lambda_A^a\overline{\lambda}_A^{n-a},s)$ have the same Euler factors outside finitely many primes.  
\end{abstract}

\maketitle


\section{Introduction}
The Langlands philosophy predicts a correspondence between certain automorphic representations and Galois representations. Moreover, the Fontaine-Mazur conjecture and its underlying philosophy specify when these Galois representations are expected to arise from the $\ell$-adic cohomology of a variety or, more generally, a motive. To each of these objects---automorphic representations, Galois representations, or motives---one can attach a natural invariant, called an $L$-function, that is a meromorphic function on some right-half complex plane. In light of these two general conjectures, one can ask: given an automorphic representation $f$, how can one construct a motive $M_f$ yielding an equality of $L$-functions $L(M_f,s)=L(f,s)$ (or at least an equality of all but finitely many Euler factors)?

In this paper, we explore this question in a very special case, namely that of algebraic Hecke characters and CM motives. In 1961, Shimura and Taniyama proved that if $A$ is a abelian variety over a number field $F$ with CM by a field $K$ linearly disjoint from $F$, then there is an algebraic Hecke character $\lambda_A$ of $FK$ such that $L(A/F, s) = L(\lambda_A, s)$ \cite{shimura-taniyama}.  Using more explicit methods, Weil had proved the same result in 1952 for factors of Jacobians of curves of the form
\[
C : y^e = \gamma x^f + \delta,
\]
for $2 \leq e \leq f$ and $\gamma, \delta \in \overline{\Q}^\times$ \cite{Weil}.

In this paper we are concerned with a converse question.  Fix $C$ as above with $(e,f)=1$ as well as a primitive $e$-th (respectively, $f$-th) root of unity $\zeta_e$ (respectively, $\zeta_f$).  Assume $F = \Q(\gamma, \delta)$ is linearly disjoint from $\Q(\zeta_e, \zeta_f)$ and that $C$ has an $F$-rational point.  Let $\lambda$ be the algebraic Hecke character associated to an isotypic CM factor of the Jacobian of $C$.  Fix a positive integer $n$ and another integer $n/2 < a \leq n$.  We explicitly construct a Chow motive over $F$ for the Hecke character $\lambda^n\overline{\lambda}^{n-a}$, where the bar denotes complex conjugation. 

More precisely, we define a group action of $G = (\Z/f\Z \times \Z/e\Z)^{n - 1}$ on the product $C^n$ that depends on the integer $a$.  There is a Chow motive $\tilde{M}$ such that any classical realization of $\tilde{M}$ is given by $G$-invariants of the corresponding realization of $C^n$.  We then decompose $\tilde{M}$ using idempotents coming from the motive $h^1(C)$.  To describe the decomposition, recall that the Jacobian of $C$ is $F$-isogeneous to a product 
\begin{equation}\label{decompose jacobian intro}
\Jac(C) \sim_F \prod_i A_i,
\end{equation} 
where each $A_i$ is an isotypic abelian variety defined over $F$ that obtains CM by a certain cyclotomic field $K_i$ upon base change.  The idempotent $e_i$ that cuts out $A_i$ can be viewed as an idempotent for $h^1(C)$, and thus $E_i = e_i^{\otimes n}$ is an idempotent for $h(C^n)$.  We show that the Chow motive $M_i \coloneqq E_i(\tilde{M})$ behaves very similarly to $A_i$.

For a finite set of primes $S$, an (incomplete) $L$-function $L^{(S)}(*,s)$ is the Euler product of the local $L$-factors in $L(*, s)$ outside of the set $S$.  Let $S_i$ denote the set of primes where the abelian variety $A_i$ has bad reduction. Let $\lambda_i : \A^\times_{FK_i}/FK_i^\times \to \C^\times$ be the algebraic Hecke character associated to $A_i$ by Weil and Shimura-Taniyama; that is, $L(A_i/F, s) = L(\lambda_i, s)$.  Then our main theorem is the following.

\begin{thm}\label{main thm}
Assume $C$ has an $F$-rational point.  Let $n$ be a positive integer and $n/2 < a \leq n$.  Assume $a=n$ if $F$ is not totally real.  For all $i$ in the decomposition \eqref{decompose jacobian intro}, there is an equality of (incomplete) $L$-functions
\[
L^{(S_i)}(M_i/F, s) = L^{(S_i)}(\lambda_i^a\overline{\lambda}_i^{n-a}, s).
\]
\end{thm}

There are a few important notes to make about the history of this problem.  First, the case when $C$ is given by $y^2 = x^3 + \delta$ ($\delta \in \Q$) was treated by Cynk and Hulek in \cite{cynk}, and our work is very much inspired by their approach.  Secondly, given a number field $k$ and any algebraic Hecke character $\lambda : \A_k^\times/k^\times \to \C^\times$, there is a standard way to construct a numerical motive $M(\lambda)$ defined over $k$ such that $L^{(S)}(M(\lambda), s) = L^{(S)}(\lambda, s)$ for some finite set of primes $S$ \cite[\S I.4]{Schappacher}.  There are two main advantages to our construction when $\lambda$ comes from one of Weil's curves, which are discussed more precisely in Section \ref{Schappacher comparison}.  First, our construction yields a Chow motive, which carries more information than a numerical motive.  Second, our construction shows that the standard motive descends to a smaller field than that given by the standard construction, in the sense that our motive $M_i$ is defined over $F$ and, when base-changed to $FK_i$, coincides with the standard motives.  Indeed, our theorem gives a positive answer to the following question in the case of algebraic Hecke characters arising from Weil curves.

\begin{question}
Let $\lambda : \A_k^\times/k^\times \to \C^\times$ be an algebraic Hecke character.  Assume there is a subfield $k' \subset k$ such that the standard motive $M(\lambda)$ descends to $k'$.  For positive integers $n, a$ as above, does $M(\lambda^a\overline{\lambda}^{n-a})$ also descend to $k'$?
\end{question}  

The structure of the paper is as follows.  In Section \ref{weil curves} we establish some facts about the curve $C$ and its Jacobian.  The main point is to find a nice basis with respect to which we can compute the Galois action on the \'etale cohomology of $C$, which is done in Proposition \ref{Frob lem}.  Then we introduce $L$-functions and use them to relate the algebraic Hecke characters in question to the matrices describing the Galois action with respect to our chosen basis.  This relationship is recorded in Corollary \ref{equality of setsgen}.

Section \ref{constructing a motive} is devoted to constructing the relevant Chow motives and calculating their $L$-functions.  We start with a brief introduction to the language and notation of motives in Section \ref{motives background}.  The group action of $G$ on $C^n$ is defined at the start of Section \ref{group action} and was inspired by similar constructions in \cite{cynk} and \cite{schreieder}.  The main technical result in the paper is Proposition \ref{two pieces of cohomologygen}, where we compute the Betti realization of the motive $\tilde{M}$ by computing the $G$-invariants of the Betti realization of $C^n$.  This was due to Schreieder \cite{schreieder} in the case when $C$ is of the form $y^2 = x^{2g+1} + 1$ and our calculation is a straightforward generalization of his result.  Theorem \ref{main thm} is proved in Proposition \ref{localLfactor2gen} and Corollary \ref{Lfn equality}.  

Finally, we briefly discuss the relationship between the motives constructed here and other motives and varieties in the literature in Section \ref{literature}.  In particular, we discuss the constructions of motives in \cite{Schappacher} and \cite{Scholl1990}. Moreover, we note in Corollary  \ref{geometric modularity} that our theorem yields modularity results for a class of smooth projective varieties constructed in \cite{schreieder}, which generalizes the modularity results of  \cite{cynk}.

\textbf{Notation:} Throughout the paper we will use $\varphi$ to denote Euler's totient function, $\varphi(n) = \#(\Z/n\Z)^\times$.  For a positive integer $n$, let $\zeta_n$ denote a primitive $n$-th root of unity.  If $X$ is a variety defined over a field $F$ and $K$ is an extension of $F$, we will write $X_{K}$ for the base change $X \times_F K$ of $X$ to $K$.  For a field $k$, we fix an algebraic closure $\overline{k}$ and write $G_k$ for the absolute Galois group $\Gal(\overline{k}/k)$.  

\section{Weil's curves}\label{weil curves}
In this section we introduce the curves $C$ studied by Weil in \cite{Weil} and give explicit descriptions of the deRham, Betti, and $\ell$-adic cohomology of $C$.  In particular, the computation of the $\ell$-adic cohomology will be given in terms of some algebraic Hecke characters, and it is the powers of these Hecke characters to which we will attach motives in Section \ref{constructing a motive}.

\subsection{The curve C}\label{curves and jacobians}
Fix integers $2\le e < f$ such that $(e,f)=1$. Let $\gamma,\delta\in \overline{\Q}^\times$ such that the field $F=\Q(\gamma,\delta)$ is linearly disjoint from the field $K=\Q(\zeta_e, \zeta_f)$.  From a notational point of view, it is easier to assume $F = \Q, e = 2$, and $f$ is prime.  We suggest the reader make these assumptions upon a first reading of the paper.  

Let $Y$ be the smooth affine curve over $F$ given by 
\[
\{y^e = \gamma x^f + \delta\}.
\]
Let $X$ be the projective closure of $Y$, which is usually singular, and $C$ the normalization of $X$.  The curve $C$ then has genus $g = \frac{(e - 1)(f - 1)}{2}$.   

After base-change to $FK$, the curve $X_{FK}$ is equipped with two automorphisms $\psi_e, \psi_f$ of orders $e, f$, given for projective coordinates $[x':y':z']$ by:
\begin{align*}
\psi_f([x' : y' : z']) &= [\zeta_f x' : y' : z'] \\
\psi_e([x ': y' : z']) &= [x' : \zeta_e y' : z'].
\end{align*}
The universal property of normalization ensures that these automorphisms extend uniquely to automorphisms of $C_{FK}$, which we shall also denote by $\psi_f$ and $\psi_e$.

\subsection{The de Rham cohomology of $C$}\label{de rham cohomology of curve}
We shall often want to consider the complex points of $C$.  Doing this requires that we base change $C$ to $\C$, which depends on a choice of embedding $F \hookrightarrow \C$.  We fix such an embedding once and for all, with the understanding that our computations depend on this choice, and write $C_\C$ for the base change of $C$ to $\C$ with respect to our fixed embedding.  

Let $\Omega_{C_\C}^1$ be the sheaf of holomorphic differential $1$-forms on $C_\C$.  There are explicit differential $1$-forms on $Y$ given by  
\[
\omega'_{i, j} = \frac{x^{i - 1}}{y^j}dx,
\]
where $1 \leq i \leq \frac{f-1}{2}, 1 \leq j \leq e-1$ if $f$ is odd and $1 \leq i \leq f-1, 1 \leq j \leq \frac{e-1}{2}$ if $f$ is even.  If $\iota : Y \hookrightarrow X$ is the natural inclusion and $N : C \to X$ is the normalization map, then write $\omega_{i, j} = N_*\iota^*(\omega'_{i,j}) \in H^0_{\dR}(C_\C, \Omega_{C_\C}^1)$. 

Note that the $\omega_{i,j}$ are defined over $F$ and thus can be viewed in the algebraic de Rham cohomology of the curve $C$.   The forms $\omega_{i, j}$ are eigenvectors for the automorphisms $\psi_e$ and $\psi_f$.  Indeed, directly from the definitions we calculate
\begin{align*}
\psi_f^*(\omega_{i,j}) &= \frac{(\zeta_f x)^{i - 1}}{y^j}d(\zeta_f x) = \zeta_f^i\omega_{i,j}\\
\psi_e^*(\omega_{i,j}) &= \frac{x^{i - 1}}{(\zeta_ey)^j}dx = \zeta_e^{-j}\omega_{i,j}.
\end{align*}

\begin{lem}\label{basis of differentials gen}
A $\C$-basis for $H^0_{\dR}(C_\C, \Omega_{C_\C}^1)$ is given by the set of forms
\begin{align*}
\bigl\{\omega_{i, j} \mid 1 \leq i \leq \frac{f-1}{2}, 1 \leq j \leq e-1\bigr\} & \text{\ \ if } f \text{ is odd,}\\
\bigl\{\omega_{i, j} \mid 1 \leq i \leq f-1, 1 \leq j \leq \frac{e-1}{2}\bigr\}  &\text{\ \ if } f \text{ is even}.
\end{align*}
\end{lem}
\begin{proof}
Since $(e,f)=1$, the forms $\omega_{i,j}$ are eigenvectors with distinct eigenvalues for  the automorphism $\psi_f^*\psi_e^*$ of $H^0_{\dR}(C_\C, \Omega^1_{C_\C})$. Hence the  $\omega_{i,j}$  are linearly independent.  Since $\dim_\C H^0_{\dR}(C_\C, \Omega^1_{C_\C}) = g = \frac{(e-1)(f-1)}{2}$, it follows that the $\omega_{i,j}$ indeed form a basis for $H^0_{\dR}(C_\C, \Omega^1_{C_\C})$
\end{proof}

To simplify notation, let us set
\[
I = \begin{cases}
\frac{f-1}{2} & f \text{ odd};\\
f-1 & f \text{ even}
\end{cases}
\hspace{2mm}
\text{ and }
\hspace{2mm}
J = \begin{cases}
e-1 & f \text{ odd};\\
\frac{e-1}{2} & f \text{ even}.
\end{cases}
\]
Thus $IJ = \frac{(e-1)(f-1)}{2} = g$ regardless of the parity of $f$.

\subsection{The Betti cohomology and Jacobian of $C$}\label{betti cohomology of curve}
In order to understand the Betti cohomology of the curve $C$ as well as the Jacobian $\Jac C$, we will make use of some Hodge theoretic terminology, which we introduce here. 

\subsubsection{Preliminaries on Hodge Theory}
A $\mathbb{Q}$-\emph{Hodge structure} $V$ of weight $w$ is a finite dimensional $\mathbb{Q}$-vector space together with a decomposition into linear subspaces $V_\mathbb{C}=\bigoplus_{p+q=w} V^{p,q},$
such that $\overline{V^{p,q}}=V^{q,p}$, where the bar denotes the action of complex conjugation. If $V$ is a $\Q$-Hodge structure, then an endomorphism $\alpha$ of $V$ is a $\Q$-vector space endomorphism of $V$ that preserves the linear subspaces $V^{p,q}$ when base changed to $\C$.

A \emph{polarization} of a $\mathbb{Q}$-Hodge structure $V$ of weight $w$ is a bilinear form $\langle\ ,\  \rangle: V\times V\rightarrow \mathbb{Q}$ that is alternating if $w$ is odd, symmetric if $w$ is even, and whose extension to $V_\mathbb{C}$ satisfies:
\begin{enumerate}
\item $\langle V^{p,q}, V^{p',q'}\rangle=0$ if $p'\ne w-p$
\item $i^{p-q}(-1)^{\frac{w(w-1)}{2}}\langle x,\overline{x}\rangle >0$ for all nonzero $x\in V^{p,q}$.
\end{enumerate}

For $X$ a smooth complex projective variety, a choice of ample line bundle on $X$ determines a polarization on the Hodge structure given by the rational cohomology $H^w(X,\mathbb{Q})$.    The category of polarizable $\mathbb{Q}$-Hodge structures is a semisimple abelian category, which we will denote by $\Q$-$\HS$. 

A $\Q$-Hodge structure has \textit{type $\{(1,0) + (0,1)\}$} if $V_\C = V^{1,0} \oplus V^{0,1}$. Note that for $X$ a smooth complex projective variety, the first rational cohomology $H^1(X,\mathbb{Q})$ is a polarizable $\Q$-Hodge structure of type $\{(1,0) + (0,1)\}$ since $H^1(X,\Q)\otimes \C=H^{1,0}(X)\oplus H^{0,1}(X)$, where $H^{i,j}(X)\cong H^j _{\dR}(X,\Omega^i_X)$ and $\Omega^i_X$ is the sheaf of holomorphic $i$ forms on $X$.  

There is an (arrow-reversing) equivalence of categories between the category of complex abelian varieties up to isogeny and the category of polarizable $\Q$-Hodge structures of type $\{(1,0) + (0,1)\}$ given by the functor $A\mapsto H^1(A,\Q)$. 

\subsubsection{The Betti Cohomology of $C$}  

We now return to our discussion of the rational Betti cohomology $H^1_B(C,\Q) \coloneqq H^1(C_\C(\C), \Q)$ of the curve $C$, which is  a $\Q$-Hodge structure of type $\{(1,0) + (0,1)\}$ by the Betti-de Rham comparison isomorphism for complex varieties.  By abuse of notation, we will consider the differential forms $\omega_{i,j}\in H^0_{\dR}(C_\C, \Omega_{C_\C}^1)$ as elements in $H^1_B(C, \Q) \otimes_\Q \C\cong H^1_B(C,\C)$.  
Furthermore, let $\omega_{f- i, e-j}$ be the image of $\overline{\omega}_{i, j}$ in $H^1_B(C, \C)$.  Namely the involution of complex conjugation acts on $H^1_B(C, \C)$ via $\omega_{i, j} \mapsto \omega_{f-i,e-j}$.  Since the Betti-de Rham comparison isomorphism is equivariant with respect to the action of $\End C$, and in particular with respect to $\psi_e$ and $\psi_f$, it follows that
\begin{equation}\label{eigenvalues}
\psi^*_f(\omega_{i,j}) = \zeta_f^i\omega_{i,j} \text{ and } \psi_e^*(\omega_{i,j}) = \zeta_e^{-j}\omega_{i,j}
\end{equation}
for all $1 \leq i \leq f-1, 1 \leq j \leq e-1$ in $H^1_B(C,\C)$.  

Now the Abel-Jacobi map yields an isomorphism of $\Q$-Hodge structures $H^1_B(C,\Q)\cong H^1_B(\Jac C,\Q)$.  We will frequently make use of this isomorphism together with the equivalence of categories discussed above to go back and forth between the language of cohomology and abelian varieties.

For a field $k \supseteq F$, let $\End_k(\Jac C)$ denote the algebra of endomorphisms of $(\Jac C)_k$ defined over $k$, and let $\End_k^0(\Jac C) = \End_k(\Jac C) \otimes_\Z \Q$.  We adopt the following conventions for the rest of the paper to simplify notation.  For a proper divisor $d | f$ (respectively, $d' | e$), write $f_d$ (respectively, $e_{d'}$) for the quotient $f/d$ (respectively, $e/d'$).  The notation $(d,d')$ will always mean $f \neq d |f$ and $e \neq d' |e$.  If we take a product or sum over $d,d'$, we mean let $d$ and $d'$ run over all proper divisors of $f$ and $e$.  Furthermore, let $K_{d,d'} = \Q(\zeta_f^d, \zeta_e^{d'})$ and $F_{d,d'} = FK_{d,d'}$.  

\begin{lem}\label{endomorphism algebra gen}
There is an embedding of $\Q$-algebras
\[
\prod_{d,d'} K_{d,d'} \hookrightarrow \End_{FK}^0(\Jac C).
\]
In particular, there is an embedding of $\Q$-algebras $\prod_{d, d'} \Q \hookrightarrow \End_F^0(\Jac C)$.
\end{lem}

\begin{proof}
Since $\psi_e$ and $\psi_f$ are endomorphisms of $C$ defined over $FK$, they induce endomorphisms $\psi_e^*, \psi_f^* \in \End_{FK}(\Jac C)$.  We want to calculate the subalgebra of $\End_{FK}^0(\Jac C)$ that they generate.  In order to do this, we observe that there is an injection (depending on our fixed embedding $F \hookrightarrow \C$, or rather, an extension of that to $FK \hookrightarrow \C$)
\[
\End_{FK}^0(\Jac C) \hookrightarrow \End_\C^0(\Jac C).
\] 
So it suffices to calculate the subalgebra of $\End_\C^0(\Jac C)$ generated by $\psi_e^*$ and $\psi_f^*$.  

Now, $(\Jac C)_\C$ is a complex abelian variety.  Since the category of complex abelian varieties up to isogeny is equivalent to the category of polarizable $\Q$-Hodge structures of type $\{(1,0)+(0,1)\}$, computing $\End_\C^0(\Jac C)$ is the same as computing $\End_{\Q-\HS}(H^1_B(\Jac C, \Q))$.  Since the Abel-Jacobi map induces an isomorphism $H^1_B(C, \Q) \cong H^1_B(\Jac C, \Q)$ in the category $\Q$-$\HS$, it suffices to compute $\End_{\Q-\HS}(H^1_B(C, \Q))$.  

By Lemma \ref{basis of differentials gen}, the eigenvalues of $\psi_e^*$ and $\psi_f^*$ acting on
$H^1_B(C,\Q) \otimes_\Q \C$ are 
\[\{\zeta_e^j : 1 \leq i \leq f-1, 1 \leq j \leq e-1\}  \text{ and  } \{\zeta_f^i : 1 \leq i \leq f-1, 1 \leq j \leq e-1\},\] 
respectively.  Thus the characteristic polynomials of $\psi_e^*$ and $\psi_f^*$ acting on $H^1(C,\Q)$ are
\[
\prod_{j = 1}^{e-1} (y - \zeta_e^j)^f = \left(\frac{y^e - 1}{y - 1}\right)^f  \text{ and }
\prod_{i = 1}^{f-1} (x - \zeta_f^i)^e = \left(\frac{x^f - 1}{x - 1}\right)^e,
\]
respectively.  Hence the minimal polynomial of $\psi_e^*$  acting on $H^1_B(C,\Q)$ is $\frac{y^e - 1}{y - 1}$, and the minimal polynomial of $\psi_f^*$ acting on $H^1_B(C,\Q)$ is $\frac{x^f - 1}{x - 1}$. It follows that the subalgebra of $\End_{\Q-\HS}(H^1_B(C, \Q))$ generated by $\psi_e^*$ is isomorphic to  $\Q[y]/(\frac{y^e - 1}{y - 1})$ and the subalgebra generated by $\psi_f^*$ is isomorphic to $\Q[x]/(\frac{x^f - 1}{x-1})$. Since $(e,f) = 1$, the polynomials $\frac{T^f - 1}{T - 1}$ and $\frac{T^e - 1}{T - 1}$ are relatively prime. Thus the subalgebra of $\End_{\Q-\HS}(H^1_B(C, \Q))$ generated by $\psi_e^*$ and $\psi_f^*$ is isomorphic to the $\Q$-algebra $\Q[x, y]/(\frac{x^f - 1}{x - 1}, \frac{y^e - 1}{y - 1})$.  The factorization of cyclotomic polynomials (see, for instance, \cite[Chapter 2]{washington}) yields the isomorphism 
\[
\Q[x,y]/(\frac{x^f - 1}{x-1}, \frac{y^e - 1}{y-1}) \cong \prod_{d, d'} K_{d,d'}.
\] 

The last sentence in the statement of the lemma follows from the first since $\End_F^0(\Jac C) = (\End_{FK}^0(\Jac C))^{\Gal(FK/K)}$.
\end{proof}

Let $e_{d_0,d'_0} \in \End_F^0(\Jac C)$ be the image, under the embedding of Lemma \ref{endomorphism algebra gen}, of the element in $\prod_{d,d'} \Q$ that has $1$ in the $(d_0,d_0')$-component and $0$ elsewhere.  Then by Lemma \ref{endomorphism algebra gen} we know that $\{e_{d,d'}\}_{d,d'}$ is an orthogonal system of idempotents in $\End_F^0(\Jac C)$.  That is, 
\begin{itemize}
\item $e_{d_1,d'_1}e_{d_2, d'_2} = \begin{cases}
e_{d_1, d'_1} & d_1 = d_2 \text{ and } d'_1 = d'_2\\
0 & \text{else}
\end{cases}$
\item $1 = \sum_{d,d'} e_{d,d'}$.
\end{itemize}
Define $A_{d,d'} = e_{d,d'}(\Jac C)$, which is an abelian variety defined over $F$.  By definition, 
\[
\Jac C \sim_F \prod_{d,d'} A_{d,d'}.
\]
Furthermore, by Lemma \ref{endomorphism algebra gen}
\[
K_{d,d'} \hookrightarrow e_{d,d'}\End_{F_{d,d'}}^0(\Jac C) = \End_{F_{d,d'}}^0(A_{d,d'}).
\]

\begin{prop}\label{decomposition of Jacobian gen}
\begin{enumerate}
\item Each $A_{d, d'}$ is an isotypic abelian variety over $F$ of dimension $g_{d,d'} = \varphi(f_d)\varphi(e_{d'})/2$ such that $(A_{d,d'})_{F_{d,d'}}$ has CM by $K_{d,d'}$. 

\item The Betti cohomology of $C$ decomposes as
\[
H^1_B(C,\Q)=\bigoplus_{d, d'} H^1_B(A_{d, d'},\Q),
\] 
where a $\C$-basis for $H^1_B(A_{d,d'}, \C)$ is given by
\begin{equation}\label{basisAdd'}
B_{d,d'}=\{\omega_{i,j} : 1 \leq i \leq f-1, 1 \leq j \leq e-1, (i, f) = d, (j, e) = d'\}.
\end{equation}
\end{enumerate}
\end{prop}

\begin{proof}
Since $K_{d,d'} \hookrightarrow \End_{F_{d,d'}}^0(A_{d,d'})$, it follows that $[K_{d,d'} : \Q] \leq 2\dim A_{d,d'}$.  On the other hand,
\[
\sum_{d,d'} 2\dim A_{d,d'} = 2\dim \Jac C = (e-1)(f-1) = \sum_{d,d'} [K_{d,d'} : \Q].
\]
Therefore we must have $\varphi(f_d)\varphi(e_{d'}) = [K_{d,d'} : \Q] = 2 \dim A_{d,d'}$, which proves the first statement.

For the last statement, recall that the Abel-Jacobi map induces 
\[
H^*_B(C, \Q) = H^*_B(\Jac C, \Q) = \bigoplus_{d,d'} H^*_B(A_{d,d'}, \Q).
\]
Let $\Phi_n$ denote the $n$-th cyclotomic polynomial.  We can identify 
\[
K_{d,d'} \cong \Q[x,y]/(\Phi_{f_d}(x), \Phi_{e_{d'}}(y))
\]
by sending $x$ to $\zeta_{f_d}$ and $y$ to $\zeta_{e_{d'}}$.  But under the embedding 
\[
\Q[x,y]/\bigl(\frac{x^f-1}{x-1}, \frac{y^e - 1}{y - 1} \bigr) \hookrightarrow \End_{FK}^0(\Jac C),
\]
we sent $x$ to $\psi_f^*$ and $y$ to $\psi_e^*$.  Therefore we have $K_{d,d'} \hookrightarrow \End_{\Q-\HS}(H^1_B(C,\Q))$ with $\zeta_{f_d} \mapsto \psi_f^*|_{A_{d,d'}}$ and $\zeta_{e_{d'}} \mapsto \psi_e^*|_{A_{d,d'}}$.  Since $\varphi_f^*|_{A_{d,d'}}$ corresponds to $\zeta_{f_d}$, it follows that $\varphi_f^*|_{A_{d,d'}}$ has order $f_d$.  Therefore the eigenvalues of $\varphi_f^*$ on $H^*_B(A_{d,d'}, \C)$ must have order equal to $f_d$.  Similarly, any eigenvalues of $\varphi_e^*$ on $H^*_B(A_{d,d'}, \C)$ must have order equal to $e_{d'}$.  Therefore $\omega_{i,j} \in V_{d,d'}$ if and only if $(i, f) = d$ and $(j, e) = d'$, which proves the last statement.  
\end{proof}

\subsection{The \'Etale Cohomology of $C$}\label{etale cohomology of curve}
Let us begin by fixing, once and for all, an algebraic closure $\overline{F}$ of $F$ as well as embeddings $K \hookrightarrow \overline{F}$ and $\overline{F} \hookrightarrow \C$.  
Let $Z_{/F}$ be a smooth projective variety.  Recall that for a given rational prime $\ell$, the fixed embedding $\overline{F} \hookrightarrow \C$ yields the identity $H^*_{\et}(Z_{\overline{F}}, \Q_\ell) = H_{\et}^*(Z_\C, \Q_\ell)$.  Moreover, for a fixed embedding $\iota_\ell\colon \overline{\Q}_\ell \hookrightarrow \C$, the Betti-\'etale comparison isomorphism for complex varieties yields an isomorphism of complex vector spaces \cite[Appendix C, 3.7]{hartshorne}: 
\[
H_{\et}^1(Z_{\overline{F}}, \Q_\ell) \otimes_{\Q_\ell, \iota_\ell} \C \cong H^1_B(Z, \C).
\]
In particular, it follows from Proposition \ref{decomposition of Jacobian gen} that we have isomorphisms
\begin{equation}\label{comparison for Ad}
\begin{tikzcd}
[column sep=tiny,row sep=small,
  ar symbol/.style = {draw=none,"\textstyle#1" description,sloped},
  isomorphic/.style = {ar symbol={\cong}},
  ]
  H^1_{\et}(C_{\overline{F}}, \Q_\ell) \otimes_{\Q_\ell, \iota_\ell} \C \ar[r, isomorphic]                         & \mathop\bigoplus\limits_{d, d'} H^1_{\et}(A_{d, d', \overline{F}}, \Q_\ell) \otimes_{\Q_\ell, \iota_\ell} \C            \\
 H^1_B(C, \C)\ar[u,isomorphic] \ar[r,isomorphic] & \mathop\bigoplus\limits_{d, d'} H^1_B(A_{d, d'}, \C), \ar[u,isomorphic] \\
\end{tikzcd}
\vspace{-4mm}
\end{equation}
that respect the decomposition given by the pairs $(d, d')$.  By abuse of notation, write $\omega_{i, j}$ for the image of $\omega_{i, j}\in B_{d,d'}$ in $H^1_{\et}(A_{d,d', \overline{F}}, \Q_\ell) \otimes_{\Q_\ell, \iota_\ell} \C$ for $1 \leq i \leq f-1, 1 \leq j \leq e-1$ and hence view $B_{d,d'}$ as a basis for $H_{\et}^1(A_{d,d', \overline{F}}, \Q_\ell) \otimes_{\Q_\ell, \iota_\ell} \C$ as well.  Therefore the involution of complex conjugation on $H^1_B(C, \C)$ can be seen on $H_{\et}^1(C_{\overline{F}}, \Q_\ell) \otimes_{\Q_\ell, \iota_\ell} \C$ by $\omega_{i, j} \mapsto \omega_{f-i,e-j}$.  Since the isomorphisms in \eqref{comparison for Ad} are equivariant with respect to the action of $\End C$, it follows that the $\omega_{i,j} \in H_{\et}^1(C_{\overline{F}}, \Q_\ell) \otimes_{\Q_\ell, \iota_\ell} \C$ have the eigenvalues calculated in \eqref{eigenvalues}.

Note that since $A_{d,d'}$ is defined over $F$, its \'etale cohomology inherits an action of the Galois group $G_F=\Gal(\overline{F}/F)$.  This action is unramified away from the (finitely many) primes of bad reduction for $A_{d,d'}$.  Fix a prime $\p$ of $F$ where $A_{d, d'}$ has good reduction.  In particular, it follows that $\p \nmid ef$.  Let $q$ denote the size of the residue field at $\p$.  That is, if $\OK_F$ is the ring of integers of $F$, then $q = \#\OK_F/\p$.  

Now fix a rational prime $\ell$ such that $\p \nmid \ell$ and an embedding $\iota_\ell\colon \overline{\Q}_\ell \hookrightarrow \C$. Consider the action of a Frobenius element $\Frob_\p^*\in G_F$ on $H^1_{\et}(A_{d,d',\overline{F}}, \Q_\ell) \otimes_{\Q_\ell, \iota_\ell} \C$.  We will show in Proposition \ref{Frob lem} below that the matrix $[\Frob_\p^*]_{B_{d,d'}}$ of $\Frob_\p^*$ with respect to the basis $B_{d,d'}$ of $H^1_{\et}(A_{d,d',\overline{F}}, \Q_\ell) \otimes_{\Q_\ell, \iota_\ell} \C$ is a generalized permutation matrix, that is, a matrix with exactly one non-zero entry in each row and column. 

Given proper divisors $d|f$ and $d'| e$, recall that $f_d=f/d$ and $e_{d'}=e/d'$.  Let $\ord_{f_d\cdot e_{d'}} q$ denote the (multiplicative) order of $q$ in $(\Z/f_d\Z)^\times \times (\Z/e_{d'}\Z)^\times$. 

\begin{prop}\label{Frob lem}
The matrix $[\Frob_\p^*]_{B_{d,d'}}$ of $\Frob_\p^*$ with respect to the basis $B_{d,d'}$ of $H^1_{\et}(A_{d,d',\overline{F}}, \Q_\ell) \otimes_{\Q_\ell, \iota_\ell} \C$ is a generalized permutation matrix on $2g_{d,d'}$ letters.  The corresponding permutation $\rho$ is a product of $\frac{\varphi(f_d)\varphi(e_{d'})}{\ord_{f_d \cdot e_{d'}} q}$ disjoint cycles, each of length $\ord_{f_d \cdot e_{d'}} q$.\end{prop}

\begin{proof}
Define the set 
\[
I_{d, d'} = \{(i, j) \in \Z/f\Z \times \Z/e\Z \mid (i, f) = d, (j, e) = d'\}.
\]
For $s = (i, j) \in I_{d,d'}$, let $i(s) \coloneqq i$ and $j(s) \coloneqq j$.  Write the action of $\Frob_\p^*$ on $H^1_{\et}(A_{d,d',\overline{F}}, \Q_\ell) \otimes_{\Q_\ell, \iota_\ell} \C$ with respect to the basis $B_{d,d'}$ as:
\[
\Frob_\p^*(\omega_t) = \sum_{s \in I_{d, d'}} a_{s,t}\omega_s,
\]
for some $a_{s,t} \in \C$.

A direct calculation shows that we have the following equality on $C(\overline{\F}_q)$ for $\nu \in \{e, f\}$:
\[
\Frob_\p \circ \psi_\nu = \psi_\nu^q \circ \Frob_\p,
\]
which implies
\[
\psi_\nu^* \circ \Frob_\p^* = \Frob_\p^* \circ (\psi_\nu^*)^q. 
\]

We compute both sides of this equality with respect to the basis $B_{d,d'}$:
\[
\psi_\nu^* \circ \Frob_\p^*(\omega_t) = \sum_{s \in I_{d,d'}} a_{s,t}\psi_\nu^*(\omega_s) = {\begin{cases} 
\sum_{s \in I_{d,d'}} a_{s,t}\zeta_f^{i(s)}\omega_s, &\mbox{if } \nu=f\\
\sum_{s \in I_{d,d'}} a_{s,t}\zeta_e^{-j(s)}\omega_s, &\mbox{if } \nu=e,
\end{cases}}
\]
whereas
\[
\Frob_\p^* \circ (\psi_\nu^*)^q(\omega_t) = \begin{cases} 
\Frob_\p^*(\zeta_f^{qi(t)}\omega_t) = \zeta_f^{qi(t)}\sum_{s \in I_{d,d'}} a_{s,t}\omega_s &  \mbox{if }\nu=f\\
\Frob_\p^*(\zeta_e^{-qj(t)}\omega_t) = \zeta_e^{-qj(t)}\sum_{s \in I_{d,d'}} a_{s,t}\omega_s & \mbox{if } \nu=e.
\end{cases}
\]

It follows that for all $s, t \in I_{d, d'}$ we must have $a_{s,t} = 0$ unless 
\[
i(s) \equiv qi(t) \bmod f \hspace{2mm}\text{ and }\hspace{2mm} j(t) \equiv qj(t) \bmod e.
\]
Since $s, t \in I_{d, d'}$, the above congruences are equivalent to the conditions
\[
\frac{i(s)}{d} \equiv \frac{qi(t)}{d} \bmod f_d\hspace{2mm}\text{ and }\hspace{2mm}  \frac{j(s)}{d'} \equiv \frac{qj(t)}{d'} \bmod e_{d'}.
\]

Since $\frac{s}{d}, \frac{t}{d} \in (\Z/f_d\Z)^\times \times (\Z/e_{d'}\Z)^\times$, it follows that $s$ determines $t$.  In other words, there is a permutation $\rho$ on $I_{d,d'}$ such that $a_{s,t} \neq 0$ if and only if $t = \rho(s)$.  This proves that $[\Frob_\p^*]_{B_{d,d'}}$ is a generalized permutation matrix.

For the claim about the structure of $\rho$, we identify $I_{d,d'}$ with $(\Z/f_d\Z)^\times \times (\Z/e_{d'}\Z)^\times$ and use the congruence condition 
\[
\left(\frac{i(s)}{d}, \frac{j(s)}{d'}\right) \equiv q\left(\frac{i(t)}{d}, \frac{j(t)}{d'}\right)
\]
in $(\Z/f_d\Z)^\times \times (\Z/e_{d'}\Z)^\times$.  Indeed, this shows that for any $\alpha \in (\Z/f_d\Z)^\times \times (\Z/e_{d'}\Z)^\times$ we have
\[
\rho(\alpha) = q^{-1}\alpha, \rho(q^{-1}\alpha) = q^{-2}\alpha, \ldots, \rho(\alpha q^{-\ord_{f_d \cdot e_{d'}} q}) = \alpha.
\]
In the above calculation, $q^{-1}$ denotes the (multiplicative) inverse of $q$ in $(\Z/f_d\Z)^\times \times (\Z/e_{d'}\Z)^\times$.  Therefore $\rho$ indeed has the desired structure.
\end{proof}

 The characteristic polynomial of a generalized permutation matrix can be calculated using basic linear algebra (see \cite[Section 1.2]{najnudel} for details).
\begin{lem}\label{gpm} 
Let $M \in \GL_N(\C)$ be a generalized permutation matrix given by the data of a permutation $\rho$ and $z_1, \ldots, z_N \in \C^{\times}$.   Write $\mathcal{C}_1,\ldots, \mathcal{C}_n$ for the supports of the disjoint cycles of $\rho$, and let $c_i = \#\mathcal{C}_i$. For $1\le m\le n$, define the complex number
\[Z_m\coloneqq \prod_{j\in \mathcal{C}_m}z_j.\]
Then the characteristic polynomial of $M$ is the polynomial
\[\prod_{m=1}^n (T^{c_m}-Z_m).\]
\end{lem}

As in the proof of Proposition \ref{Frob lem}, let us write the matrix of $\Frob_\p$ with respect to $B_{d,d'}$ as $(a_{s, \rho(s)})_{s \in I_{d,d'}}$ for the permutation $\rho$.  Let $q^{-1}$ denote the (multiplicative) inverse of $q$ in $(\Z/f_d\Z)^\times \times (\Z/e_{d'}\Z)^\times$.  For a fixed element $b \in (\Z/f_d\Z)^\times \times (\Z/e_{d'}\Z)^\times$, define
\begin{equation}\label{defining Zs}
Z_b = \prod_{i = 1}^{\ord_{f_d \cdot e_{d'}} q} a_{bq^{-i}, \rho(bq^{-i})},
\end{equation}
where we have identified $I_{d,d'} \cong (\Z/f_d\Z)^\times \times (\Z/e_{d'}\Z)^\times$ via $(i, j) \leftrightarrow (i/d, j/d')$.  In other words, $Z_b$ is the product over all the nonzero entries in $[\Frob_\p^*]_{B_{d,d'}}$ corresponding to the disjoint cycle in $\rho$ containing $b$.

\begin{cor}\label{characteristic poly}
The characteristic polynomial of $\Frob_\p^*$ acting on $H^1_{\et}(A_{d,d',\overline{F}}, \Q_\ell)$ is
\[
\prod_{c = 1}^{\frac{\varphi(f_d)}{\ord_{f_d} q}}\prod_{c' = 1}^{\frac{\varphi(e_{d'})}{\ord_{e_{d'}} q}} \left(T^{\ord_{f_d \cdot e_{d'}} q} - Z_{(c,c')}\right).
\]
\end{cor}
\begin{proof}
This follows directly from Proposition \ref{Frob lem} and Lemma \ref{gpm} since the characteristic polynomial of $\Frob_\p^*$ can be computed after any base change.
\end{proof}

\begin{rem}
Note that while the $a_{s,t}$ may be complex numbers, Corollary \ref{characteristic poly} implies that the $Z_b$ are algebraic since the characteristic polynomial of $\Frob_\p^*$ vanishes at all of the $(\ord_{f_d \cdot e_{d'}} q)$-th roots of $Z_b$. 
\end{rem}

We shall also want to understand the action of complex conjugation on $Z_b$.  Unfortunately, we can only do this when $F$ is totally real.  Letting the $^{-}$ symbol denote complex conjugation, we record the following proposition for later use.

\begin{prop}\label{complex conjugation action}
Suppose that $F$ is totally real, and let $\p$ be a prime of $F$ where $A_{d,d'}$ has good reduction.  Then for any $s, t \in I_{d,d'}$ we have $\overline{a_{s,t}} = a_{-s,-t}$.  In particular, for any $b \in (\Z/f_d\Z)^\times \times (\Z/e_{d'}\Z)^\times$ we have $\overline{Z_b} = Z_{-b}$.  
\end{prop}

\begin{proof}
Since $F$ is totally real, there is a well-defined complex conjugation $c \in G_F$, and for $\p$ as in the statement of the proposition, we have the relation $c \circ \Frob_\p = \Frob_\p \circ c$.  Hence $\Frob_\p^* \circ c^* = c^* \circ \Frob_\p^*$.  We compute both sides of this relation, making use of the fact that the \'etale-Betti comparison isomorphism under which we have identified the various incarnations of $\omega_t$ is equivariant with respect to complex conjugation.  Thus for any $t\in I_{d,d'}$, we know that $c^*(\omega_t) = \overline{\omega}_t = \omega_{-t}$.  Hence we have
\begin{align*}
\Frob_\p^* \circ c^*(\omega_t) &= \Frob_\p(\omega_{-t}) = \sum_{s \in I_{d,d'}}a_{s,-t}\omega_s\\
c^* \circ \Frob_\p^*(\omega_t) &= c^*\left(\sum_{s \in I_{d,d'}} a_{s,t}\omega_s\right) = \sum_{s \in I_{d,d'}} \overline{a_{s,t}}\overline{\omega}_s = \sum_{s \in I_{d,d'}}\overline{a_{-s,t}}\omega_s.
\end{align*}
This gives $\overline{a_{s,-t}} = a_{-s,t}$, as desired.  The fact that $\overline{Z}_b = Z_{-b}$ now follows directly from the definition of $Z_b$.
\end{proof}

\subsection{The $L$-function of $C$}\label{lfactors}
Let $k$ and $E$ be number fields.  Write $\OK_k$ for the ring of integers of $k$.  The language of compatible systems of Galois representations is useful for defining the $L$-functions of interest.  Recall that a $d$-dimensional \textit{compatible system} of $G_k$-representations $V = \{V_\lambda\}_\lambda$ is a collection of representations $\{\rho_\lambda : G_k \to \GL_d(E_\lambda)\}_{\lambda \text{ prime of } E}$ such that there is a finite set $S$ of primes of $k$ such that:
\begin{enumerate}
\item $\rho_\lambda$ is unramified outside $S$ and the primes of $k$ lying over $\ell = \lambda \cap \Z$
\item for all primes $\p$ of $k$ outside $S$, there is a polynomial $F_\p(X) \in E[X]$ such that, for all primes $\lambda$ of $E$ such that $\lambda \cap \Z \neq \p \cap \Z$, the characteristic polynomial $\rho_\lambda(\Frob_\p)$ is equal to $F_\p(X)$, independent of $\lambda$.
\end{enumerate}
For example, if $Z/k$ is a smooth projective variety, then $\{H^n_{\et}(Z_{\overline{k}}, \Q_\ell)\}_\ell$ is a compatible system of $G_k$-representations with $E = \Q$ for any $n$.  If $M$ is a Chow motive defined over $k$ (see Section \ref{motives background} for the definition), then $\{H^n_{\et}(M_{\overline{k}}, \Q_\ell)\}_\ell$ is a system of $G_k$-representations.  It is not known in general whether the system is compatible.  However, for the motives we construct, we prove that this system of Galois representations is compatible in Proposition \ref{localLfactor2gen}.

The $L$-function of a compatible system $V$ of $d$-dimensional $G_k$-representations is a complex analytic (or meromorphic) function that encodes the data of how the local Galois groups $G_{k_\p} \hookrightarrow G_k$ act on the Galois representations $\{V_\lambda\}_\lambda$.  (Here $\p$ is any prime of $k$ and $\lambda$ is chosen such that $\p \cap \Z\ne \lambda \cap \Z$.)  It is defined as an Euler product with one local factor for each prime of $k$.  It is easiest to define these local factors when the representation is unramified at $\p$.  As these are the only factors that will concern us, we restrict our definition to that case.

\begin{defn}\label{local Lfactor variety}
Let $\p$ be a prime of $k$ at which a $d$-dimensional compatible system of Galois representations $V$ is unramified.  Choose a prime $\lambda$ of $E$ such that $\p \cap \Z\ne \lambda \cap \Z$.  Let $P_\p(T)$ be the characteristic polynomial of $\Frob_\p$ on $V$ (which is independent of the choice of $\lambda$).  Let $q_\p = \#\OK_k/\p$.  The \textit{local $L$-factor of $V$ at $\p$} is 
\[
L_\p(V/k, s) \coloneqq q_\p^{-2ds}P_\p(q_\p^s),
\]
which is a polynomial of degree $d$ in $q_\p^{-s}$.  
\end{defn}  
If $Z/k$ is a variety and $V = \{H^*_{\et}(Z_{\overline{k}}, \Q_\ell)\}_\ell$, then we write $L_\p(Z/k, s)$ and call it the \textit{local $L$-factor of $Z$ at $\p$}.  This is the (incomplete) Hasse-Weil zeta function of $Z$.  If $Z = A$ is an abelian variety, then $H^r_{\et}(A_{\overline{k}}, \Q_\ell) = \Lambda^r H^1_{\et}(A_{\overline{k}}, \Q_\ell)$.  Therefore we abuse notation slightly and write $L(A/k, s)$ for the $L$-function of $\{H^1_{\et}(A_{\overline{k}}, \Q_\ell)\}_\ell$ in the case of abelian varieties.  This should not cause any confusion.

Note that given a local $L$-factor $L_\p(V/k, s)$, it is possible to recover $P_\p(T)$ by replacing $q_\p^{-s}$ in $L_\p(V/k, s)$ with a variable $T$ and multiplying the resulting Laurent polynomial by $T^d$.  We shall often switch between local $L$-factors and characteristic polynomials in what follows.  In particular, we can restate Corollary \ref{characteristic poly} as follows.

\begin{cor}\label{localLfactor1gen}
If $\p$ is a prime of $F$ where $A_{d, d'}$ has good reduction, then the local $L$-factor at $\p$ of $A_{d,d'}$ over $F$ is 
\[
L_\p(A_{d,d'}/F, s) = \prod_{c = 1}^{\frac{\varphi(f_d)}{\ord_{f_d} q}}\prod_{c' = 1}^{\frac{\varphi(e_{d'})}{\ord_{e_{d'}} q}} \left(1 - Z_{(c,c')}q^{-s\ord_{f_d \cdot e_{d'}}q}\right).
\]
\end{cor}

\begin{proof}
Recall that the compatible system of Galois representations associated to an abelian variety is ramified at exactly the primes of bad reduction.  The result now follows from Definition \ref{local Lfactor variety} and Corollary \ref{characteristic poly}.
\end{proof}

As we shall only be concerned with the local factors where our representations are unramified, we shall define the incomplete $L$-function of a compatible system $V$ as follows.  Let $S$ be the finite set of primes of $k$ at which $V$ is ramified.  The incomplete $L$-function is
\[
L^{(S)}(V/k, s) \coloneqq \prod_{\p \not\in S} L_\p(V/k, s)^{-1}.
\]
It is important to note that the $L$-function defined above depends on the field $k$; it will change if we replace $k$ by a larger field $k'$.

Weil proved \cite{Weil} that the varieties $A_{d,d'}$ have algebraic Hecke characters associated to them, a fact that was later generalized to all CM abelian varieties by Shimura and Taniyama \cite{shimura-taniyama}.  In order to state this correspondence, let us recall some definitions.  For a number field $k$, write $\A_k$ for the ring of adeles of $k$ and $\OK_v$ for the completion of the ring of integers of $k$ at a finite place $v$.  For a place $v$ of $k$, write $\iota_v : k_v^\times \hookrightarrow \A_k^\times$ for the embedding sending $x \in k_v^\times$ to the idele with $x$ in the $v$-th component and $1$ elsewhere. 

\begin{defn}\label{Hecke character defn}
A \textit{Hecke character} is a continuous homomorphism $\lambda : \A_k^\times/k^\times \to \C^\times$.  Such a character is said to be \textit{algebraic} if for every archimedean place $\sigma$ of $k$, there exists $n_\sigma \in \Z$ (respectively, $n_\sigma, m_\sigma \in \Z$) if $\sigma$ is real (respectively, complex), such that $\lambda(\iota_\sigma(x)) = x^{n_\sigma}$ (respectively, $\lambda(\iota_\sigma(x)) = x^{n_\sigma}\bar{x}^{m_\sigma}$) for all $x \in k_\sigma^\times$.  The \textit{conductor} of $\lambda$ is the largest ideal $\m$ of $k$ such that $\lambda$ is trivial on $\prod_{v} 1 + \m \OK_v$.
\end{defn}

Fix a uniformizer $\varpi_v$ of $\OK_v$ for each finite place $v$ of $k$, and let $\p_v$ be the corresponding prime ideal of $F$.  Write
\[
\lambda(\p_v) = \begin{cases}
\lambda(1, \ldots, 1, \varpi_v, 1, \ldots, 1) &\mbox{if } \p_v \nmid \m\\
0 & \mbox{otherwise}.
\end{cases}
\]

\begin{defn}
With notation as in Definition \ref{Hecke character defn}, let $\p \nmid \m$ be a prime of $k$ with residue degree $q_\p$.  The \textit{local $L$-factor of $\lambda$ at $\p$} is
\[
L_\p(\lambda, s) \coloneqq 1 - \lambda(\p)q_\p^{-s}.
\] 
The {incomplete $L$-function of $\lambda$} is 
\[
L^{(\m)}(\lambda, s) \coloneqq \prod_{\p \nmid \m} L_\p(\lambda, s)^{-1}. 
\]
\end{defn}

Recall that $A_{d,d'}$ is an abelian variety over $F$ such that $(A_{d,d'})_{F_{d,d'}}$ has CM by $K_{d,d'} = \Q(\zeta_f^d\zeta_e^{d'})$.  Since $F$ is linearly disjoint from $K$ by assumption, it follows that $\Gal(F_{d, d'}/F) \cong (\Z/f_d\Z \times \Z/e_{d'}\Z)^\times$ acts transitively on the set of embeddings $\{\iota : K_{d,d'} \hookrightarrow \C\}$.  Therefore there is an algebraic Hecke character 
\[\lambda_{d,d'} = \lambda_{A_{d,d'}/F_{d,d'}} : \A_{F_{d,d'}}^\times/F_{d,d'}^\times \to \C^\times\]
 such that we have an equality of $L$-functions \cite[Theorem 12]{Shimura},
\begin{equation}\label{equality of Lfnsgen}
L(A_{d,d'}/F, s) = L(\lambda_{d,d'}, s).
\end{equation}
Furthermore, the support of the conductor of $\lambda_{d,d'}$ is exactly the set of primes of $F_{d,d'}$ where $(A_{d,d'})_{F_{d,d'}}$ has bad reduction \cite[\S 7, Corollary 1]{Serre-Tate}.  

Recall that for a positive integer $n$, if $p$ is a prime not dividing $n$, then $p$ splits into $\varphi(n)/\ord_n p$ distinct primes in $\Q(\zeta_n)$, each of residue class degree $\ord_n p$ \cite[Theorem 2.13]{washington}.  In particular, $p$ splits into $\varphi(f_d \cdot e_{d'})/\ord_{f_d\cdot e_{d'}} p$ primes in $K_{d,d'}$, each of residue class degree $\ord_{f_d \cdot e_{d'}} p$.  Since $F$ is linearly disjoint from $K_{d, d'}$, it follows that each prime $\p$ of $F$ lying over $p$ splits into $\varphi(f_d \cdot e_{d'})/\ord_{f_d \cdot e_{d'}}p$ primes in $F_{d,d'}$, each of residue class degree $\ord_{f_d \cdot e_{d'}} p$.  Let $f_\p$ be the residue degree of $\p$ over $p$; that is, $q = p^{f_\p}$ using the notation introduced just before Proposition \ref{Frob lem}.  Then the local $L$-factor of $L(\lambda_{d,d'}, s)$ at $\p$ is
\begin{equation}\label{LfnOfPsigen}
\prod_{\p' |\p}\left(1 - \lambda_{d,d'}(\p')q^{-s\ord_{f_d \cdot e_{d'}}p}\right),
\end{equation}
where the product runs over all primes $\p'$ of $F_{d,d'}$ lying over $\p$.  

Note that since $q = p^{f_\p}$, it follows that $\ord_{f_d \cdot e_{d'}} q$ divides $\ord_{f_d \cdot e_{d'}} p$.  Write 
\[
\ord_{f_d \cdot e_{d'}} p = m \cdot \ord_{f_d \cdot e_{d'}} q.
\]
For each prime $\p'$ of $F_{d,d'}$ lying over $\p$, fix an $m$-th root of $\lambda_{d,d'}(\p')$, and let $\mu_m$ denote the group of $m$-th roots of unity.  Then we have the following corollary.

\begin{cor}\label{equality of setsgen}
For primes $\p$ of $F$ as above, we have an equality of sets of $\frac{\varphi(f_d)\varphi(e_{d'})}{\ord_{f_d \cdot e_{d'}} q}$ elements:
\[
\left\{Z_{(c,c')}\  \middle|\ 1 \leq c \leq \frac{\varphi(f_d)}{\ord_{f_d} q}, 1 \leq c' \leq \frac{\varphi(e_{d'})}{\ord_{e_{d'}} q}\right\} =\left \{\zeta_m\sqrt[m]{\lambda_{d,d'}(\p')} \ \middle|\  \p' | \p, \zeta_m \in \mu_m\right\},
\]
where in the second set $\p'$ runs over primes of $F_{d,d'}$ lying over $\p$.  Furthermore, if $F$ is totally real then $(Z_b)^m = \lambda_{d,d'}(\p')$ implies $(Z_{-b})^m = \overline{\lambda}_{d,d'}(\p')$.
\end{cor}

\begin{proof}
The first statement follows directly from Corollary \ref{localLfactor1gen} and equations \eqref{equality of Lfnsgen} and \eqref{LfnOfPsigen}.  For the second equality, we have $\lambda_{d,d'}(\p') = (Z_b)^m$ implies
\[
\overline{\lambda}_{d,d'}(\p') = \overline{(Z_b)}^m = Z_{-b}^m
\]
by Proposition \ref{complex conjugation action}.
\end{proof}

\section{Constructing a motive attached to powers of $\lambda$}\label{constructing a motive}
Let $\lambda = \lambda_{d,d'}$ be the Hecke character attached to $A_{d,d'}$ obtained as in the previous section.  Fix a positive integer $n$ and $n/2 < a \leq n$.  The goal of this section is to construct a Chow motive $M_{d,d'}$ defined over $F$ such that $L^{(S_{d,d'})}(M_{d,d'}/F, s) = L^{(S_{d,d'})}(\lambda^a\overline{\lambda}^{n-a}, s)$.  In \S \ref{motives background} we introduce some notation and recall some background about motives.  The key ingredient in the construction of $M_{d,d'}$ is a group action on $C^n$.  We describe the action in \S \ref{group action} and then construct $M_{d,d'}$ using the invariants of the group action and the idempotent $e_{d,d'}$ defined right before Proposition \ref{decomposition of Jacobian gen}.  Using an explicit calculation of the Betti realization of $M_{d,d'}$ (Corollary \ref{betti motive basis}), we compute the $L$-function of $M_{d,d'}$ (Proposition \ref{localLfactor2gen}) and match it with that of $\lambda_{d,d'}^a\overline{\lambda}_{d,d'}^{n-a}$ (Corollary \ref{Lfn equality}).  Finally, Section \ref{literature} summarizes the relationship between the motives $M_{d,d'}$ constructed in this paper and other motives associated to Hecke characters in the literature, especially the standard motive of a Hecke character as found in \cite[\S I.4]{Schappacher}.  

\subsection{Background on motives}\label{motives background}
Let $k$ be a number field.  We begin by briefly recalling the construction of the category of Chow motives over $k$, closely following \cite{Andre}.  Let $\V_k$ denote the category of smooth projective varieties over $k$.  To any object $Z$ in $\V_k$, let $\cZ^n(Z)$ denote the group of algebraic cycles in $Z$ of codimension $n$ and $\cZ^n(Z)_\Q = \cZ^n(Z) \otimes_\Z \Q$.  If $\sim$ is an adequate equivalence relation on algebraic cycles \cite[Definition 3.1.1.1]{Andre}, write $\cZ_\sim^n(Z)_\Q = \cZ^n(Z)_\Q/\sim$.   

For a fixed adequate equivalence relation as above, define a category $\CV^\sim_k$ whose objects are smooth projective varieties and 
\[
\Hom_{\CV^\sim_k}(X, Y) \coloneqq \cZ_\sim^{\dim X}(X \times_k Y)_\Q.
\]
See \cite[\S 3.1.3]{Andre} for the definition of composition of morphisms.  The category $\CV^\sim_k$ is a tensor category.  In particular, it is an additive category with a tensor product structure that is symmetric.  The category of motives $\M^\sim_k$ is defined to be the pseudoabelian closure of $\CV^\sim_k$.  

Objects in $\M_k^\sim$ can be represented explicitly as triples $(X, p, m)$, where $X$ is a smooth projective variety, $p \in \End_{\CV^\sim_k}(X)$ such that $p^2=p$, and $m \in \Z$.  Morphisms of motives are given by 
\[
\Hom_{\M^\sim_k}((X, p, n), (Y, q, m)) = q \circ \cZ_{\sim}^{\dim X -n+m}(X \times_k Y)_\Q \circ p.
\]
Furthermore, there is a natural contravariant functor $h : \V_k \to \M_k$ given by $h(X) = (X, \Delta_X, 0)$, where $\Delta_X$ is the diagonal subvariety of $X \times_k X$.

The adequate equivalence relations $\sim$ that will be of interest for us are rational $\sim_{\rat}$ \cite[\S 3.2.2]{Andre}, homological $\sim_{\mhom}$ \cite[\S 3.3.4]{Andre}, and numerical $\sim_{\num}$ \cite[\S 3.2.7]{Andre} equivalence, listed here in decreasing order of fineness.  The corresponding categories of motives will be denoted by $\M^{\rat}_k$, $\M^{\hom}_k$, and $\M^{\num}_k$.  The objects in $\M^{\rat}_k$ are called Chow motives and they are universal in the sense that $\sim_{\rat}$ is the finest possible adequate equivalence relation \cite[Lemme 3.2.2.1]{Andre}.   

\subsubsection{The motive $h^1$ of a curve}\label{h1 background}  Let $C \in \V_k$ be a geometrically connected curve, and assume that $C$ has a $k$-rational point $P$.  Let $p_0$ be the cycle on $C\times C$ given by $\{P\} \times C$ and $p_2$ the cycle given by $C \times \{P\}$.  Then $p_0$ and $p_2$ are idempotent for any choice of adequate equivalence relation.  Define $p_1 = 1 - p_0 - p_2$, which is also idempotent.  Let $h^1(C) = (C, p_1, 0)$; that is, $h^1(C)$ is the image of $p_1$.  It is well defined up to unique isomorphism and we have \cite[Proposition 3.3]{Scholl}:
\[
\End_{\M_k^\sim}(h^1(C)) = \End_k^0(\Jac C).
\]
In particular, $\End_{\M_k^\sim}(h^1(C))$ is independent of the choice of adequate equivalence relation $\sim$.

Note that since $\Jac C$ is isogenous to a product of isotypic abelian varieties, $\Jac C \sim \prod_{i \in I} A_i$, it follows that there is an orthogonal system of idempotents $\{e_i\}_{i \in I} \in \End_{\M_k^\sim}(h^1(C))$ corresponding to the decomposition of $\Jac C$.  For any positive integer $n$ it follows that $\{e_{i_1} \otimes \cdots \otimes e_{i_n} : i_j \in I\}$ is an orthogonal system of idempotents in $\End_{\M_k^\sim}(h^1(C)^{\otimes n})$.  In particular, when $C$ is one of the Weil curves introduced in Section \ref{weil curves} we have an orthogonal system of idempotents in $\End_{\M_F^{\sim}}(h^1(C))$ corresponding to the $e_{d,d'}$ introduced prior to Proposition \ref{decomposition of Jacobian gen}.  We continue to write $e_{d,d'}$ for the corresponding element of $\End_{\M_k^{\sim}}(h^1(C))$.  Furthermore, we can view $e_{d,d'}$ as an element in $\End_{\M_k^\sim}(h(C))$ by extending by $0$ on $h^0(C) := (C, p_0, 0)$ and $h^2(C) := (C, p_2, 0)$.  

\subsubsection{Realization functors} There are covariant functors on $\M_k^{\rat}$ corresponding to the various classical Weil cohomology theories.  For example, the Betti realization is a functor
\[
H_B^*(\cdot, \Q) \colon \M_k^{\rat} \to \Q\text{-}\HS.
\]
For any $X \in \V_k$ we have $H^*_B(h(X), \Q) = H^*_B(X, \Q)$.  Write $k$-$\FVS$ for the category of finite-dimensional $k$-vector spaces with a filtration.  The de Rham realization is a functor
\[
H_{\dR}^*\colon \M_k^{\rat} \to k\text{-}\FVS.
\]
For any $X \in \V_k$ we have $H^*_{\dR}(h(X)) = H_{\dR}^*(X)$, where $H_{\dR}^*(X)$ denotes the algebraic de Rham cohomology of the $k$-scheme $X$.  For any rational prime number $\ell$, let $\Rep_{\Q_\ell}(G_k)$ denote the category of continuous representations of $G_k$ on finite dimensional $\Q_\ell$-vector spaces.  The $\ell$-adic realization is a functor
\[
H_{\et}^*(\cdot, \Q_\ell) \colon \M_k^{\rat} \to \Rep_{\Q_\ell}(G_k).
\]
For any $X \in \V_k$ we have $H^*_{\et}(h(X), \Q_\ell) = H_{\et}^*(X_{\overline{k}}, \Q_\ell)$.  If this system of $\ell$-adic representations is compatible, then we define the $L$-function of a Chow motive as the $L$-function of its $\ell$-adic realizations. 

\subsubsection{Galois descent} There is a theory of Galois descent for motives.  Let $k/k'$ be a finite Galois extension.  There is a base change functor $\mathord{-} \times_{k'} k \colon \M^\sim_{k'} \to \M^\sim_k$ as well as a Weil restriction of scalars functor $\Res_{k/k'} \colon \M_k^\sim \to \M_{k'}^\sim$.  We shall say that a motive $M \in \M_k^\sim$ \textit{descends} to $k'$ if there exists $M' \in \M_{k'}^\sim$ such that $M = M' \times_{k'} k$.  Given $M \in \M_{k'}^\sim$, the motive $M \times_{k'} k$ has a natural action of $\Gal(k/k')$.  A submotive $N$ of $M \times_{k'} k$ descends to $k$ if and only if it is stable under the action of $\Gal(k/k')$ on $M \times_{k'} k$ \cite[1.16, Lemma 1.17]{Scholl}.

\subsubsection{Group actions on varieties} There is a way to construct a motive $h(X)^G \in \M_k^{\rat}$ from the action of a finite group $G$ on a smooth projective variety $X \in \V_k$, assuming that the action of $G$ on $X$ is defined over $k$.  Explicitly, this motive can be written as
\begin{equation}\label{G-invariant motive}
h(X)^G = (X, \frac{1}{|G|}\sum_{g \in G}\Gamma_g, 0),
\end{equation}
where $\Gamma_g \in \cZ_{\rat}^{\dim X}(X \times_k X)_\Q$ is the transpose of the graph of the automorphism $g$ \cite[proof of Proposition 1.2]{Rollin-Aznar}.  In coordinates,
\[
\Gamma_g = \{(gx, x) \in X \times_k X \}.
\]
Thus $\Gamma_{1_G} = \Delta_X$.  Furthermore, this construction behaves well with respect to realizations in the following sense.  Each realization of $X$ inherits an action of $G$.  Any realization of $h(X)^G$ is just the $G$-invariant vectors in the corresponding realization of $X$.  Finally, the motive $h(X)^G$ descends to a subfield $k' \subseteq k$ if both $X$ and the cycle $\frac{1}{|G|}\sum_{g \in G} \Gamma_g$ descend to $k'$.

\subsection{The group action on $C^n$ and the Chow motives}\label{group action}
We now return to the notation from Section \ref{weil curves}.  Recall that for proper divisors $d|f$, $d'|e$, the abelian variety $A_{d,d'}$ is defined over $F$ and has CM by $K_{d,d'}$ after base change to $F_{d,d'}$. Consider the $n$-fold product $C_{F_{1,1}}^n$. Fix another positive integer $a$ such that $n/2 <a \le n$.   Everything that follows will depend on the choice of integers $a$ and $n$, though this will not be reflected in the notation.  Consider the subgroup $G$ of $\Aut C_{F_{1,1}}^n$ given by
\[
G = \left\{\prod_{k = 1}^n \psi_f^{u_k}\psi_e^{v_k}\  \middle| \  \begin{subarray}{c}u_1 + \cdots + u_a - u_{a + 1} - \cdots - u_n \equiv 0 \bmod f,\\  v_1 + \cdots + v_a - v_{a + 1} - \cdots - v_n \equiv 0 \bmod e\end{subarray}\right\},
\]
where the $k$-th factor $\psi_f^{u_k}\psi_e^{v_k}$ in the product $\prod_{k = 1}^n \psi_f^{u_k}\psi_e^{v_k}$ acts on the $k$-th factor in the product $C_{F_{1,1}}^n$.  Note that we have an isomorphism of groups $G \cong (\Z/f\Z \times \Z/e\Z)^{n - 1}$.  Write 
\[\varepsilon = \frac{1}{|G|}\sum_{g \in G} \Gamma_g \in \End_{\M_{F_{1,1}}^{\rat}}(h(C^n_{F_{1,1}}))\]
 for the idempotent cutting out $h(C_{F_{1,1}}^n)^G$.  We now use Galois descent to produce a motive defined over the field $F$ instead of $F_{1,1}$. 

\begin{lem}\label{galois descent of our motive}
There is a motive $M \in \M^{\rat}_F$ such that $M \times_F F_{1,1} = h(C^n_{F_{1,1}})^G$.
\end{lem}

\begin{proof}
Recall that the motive $h(C^n_{F_{1,1}})^G$  consists of the data  $(C^n_{F_{1,1}}, \varepsilon, 0)$.  Since the curve $C$ is defined over the field $F$, it suffices to check that
\[
\sum_{g \in G} \Gamma_g \subset (C^n \times_F C^n) \times_F F_{1,1}
\]
is $\Gal(F_{1,1}/F)$-stable. 

Let $[a:b:c]$ be projective coordinates of the curve $C$. Recall that $C$ is the normalization of the singular projective curve $X$ given by $(y')^ez'^{f-e}=\gamma (x')^f+\delta (z')^f$.  After base change to $F_{1,1}$ we have $\psi_f([x' : y' : z']) = [\zeta_f x' : y' : z']$ and 
$\psi_e([x ': y' : z']) = [x' : \zeta_e y' : z']$. Since resolving the singular point $[0:1:0]$ of $X$ will result in new coordinates obtained as rational functions in the variables $x',y',z'$, both $\psi_f$ and $\psi_e$ will act by powers of $\zeta_e$ and $\zeta_f$, respectively, on the coordinates $a$, $b$, and $c$ of $C_{F_{1,1}}$. 

Hence there are integers $m_i$, $n_j$ for $1\le i,j \le 3$ such that for any point $[a:b:c]$ of $C_{F_{1,1}}$ we have $\psi_f([a:b:c])=[\zeta_f^{m_1}a:\zeta_f^{m_2}b:\zeta_f^{m_3}c]$ and $\psi_e([a:b:c])=[\zeta_e^{n_1}a:\zeta_e^{n_2}b:\zeta_e^{n_3}c]$.

Consider the element $g_0= \prod_{k = 1}\psi_f^{u_k}\psi_e^{v_k} \in G$. Then the class $\Gamma_{g_0} \subset (C^n \times_F C^n) \times_F F_{1,1}$ is given by
\[
\Gamma_{g_0} = \left\{\prod_{k=1}^n[\zeta_f^{m_1u_k}\zeta_e^{n_1v_k}a_k: \zeta_f^{m_2u_k}\zeta_e^{n_2v_k}b_k: \zeta_f^{m_3u_k}\zeta_e^{n_3v_k}c_k] \times \prod_{k=1}^n[a_k: b_k:c_k]\ \middle| \  [a_k: b_k:c_k] \in C_{F_{1,1}}\right\}.
\]
Let $\tau \in \Gal(F_{1,1}/F) \cong \Gal(\Q(\zeta_f,\zeta_e)/\Q)$.  Say $\zeta_f^\tau = \zeta_f^i$ and $\zeta_e^\tau = \zeta_e^j$.  Then for an element $\alpha=\prod_{k=1}^n[\zeta_f^{u_km_1}\zeta_e^{v_kn_1}a_k: \zeta_f^{u_km_2}\zeta_e^{v_kn_2}b_k: \zeta_f^{u_km_3}\zeta_e^{v_kn_3}c_k] \times \prod_{k=1}^n[a_k: b_k:c_k]$ of  $\Gamma_{g_0}$, we have 
\begin{align*}
\alpha^\tau&=
\prod_{k=1}^n[\zeta_f^{u_km_1}\zeta_e^{v_kn_1}a_k: \zeta_f^{u_km_2}\zeta_e^{v_kn_2}b_k: \zeta_f^{u_km_3}\zeta_e^{v_kn_3}c_k]^\tau\times \prod_{k=1}^n[a_k: b_k:c_k]^\tau\\
&=\prod_{k=1}^n[\zeta_f^{iu_km_1}\zeta_e^{jv_kn_1}a_k^\tau: \zeta_f^{iu_km_2}\zeta_e^{jv_kn_2}b_k^\tau: \zeta_f^{iu_km_3}\zeta_e^{jv_kn_3}c_k^\tau] \times \prod_{k=1}^n[a_k^\tau: b_k^\tau:c_k^\tau].
\end{align*}
Namely, we have $\alpha^\tau\in \Gamma_{g'}$, where $g' = \prod_{k = 1}^n \psi_f^{iu_k}\psi_e^{jv_k} \in G$.  Thus $\sum_{g \in G} \Gamma_g$ is $\Gal(F_{1,1}/F)$-stable, as desired.
\end{proof}

We can now define the Chow motive associated to $\lambda_{d,d'}^a\overline{\lambda}_{d,d'}^{n-a}$.  Let $E_{d,d'} = e_{d,d'}^{\otimes n} \in \End_{\M_F^{\rat}}(h(C^n))$ as in Section \ref{h1 background}.  Then by Lemma \ref{galois descent of our motive} we may compose $E_{d,d'}$ with $\varepsilon$ in $\End_{\M_F^{\rat}}(h(C^n))$.  Define
\[
M_{d,d'} = (E_{d,d'} \circ \varepsilon)(h(C^n)) \in \M_F^{\rat},
\]
and define
\begin{equation}\label{definition of M}
\tilde{M} = \bigoplus_{d,d'} M_{d,d'} = (\sum_{d,d'} E_{d,d'} \circ \varepsilon)(h(C^n)) \in \M_F^{\rat}.
\end{equation}
We reiterate that $M_{d,d'}$ depends on the choice of $a$ and $n$ even though this is not reflected in the notation.

\subsection{Computing the cohomology of the motives $M_{d,d'}$}
We begin by studying the Betti realization of $M$.  Once we understand this, we will then determine its decomposition with respect to the idempotents induced by the $E_{d,d'}$.  We begin by establishing some notation.  Let $\pi_k : C^n \to C$ be the natural projection map onto the $k$-th copy of $C$ in $C^n$.  Write $\Omega$ for the fundamental class of $C$ and $\Omega_k = \pi_k^*(\Omega)$.  In particular, $\Omega$ represents a nonzero class in $H^{1,1}(C) = H^2_B(C, \C)$ under the Betti-de Rham comparison isomorphism.  Let $\omega_{k, i, j} = \pi_k^*(\omega_{i,j})$ for all $1 \leq i \leq I, 1 \leq j \leq J$.  For the same range of $i, j$, define
\[
\Sigma_{i,j} = \omega_{1, i, j}\cdots\omega_{a, i, j}\bar{\omega}_{a+1, i, j}\cdots\bar{\omega}_{n, i, j}.
\]
As with the $\omega_{i,j}$, for $1 \leq i \leq I, 1 \leq j \leq J$, define
\[
\Sigma_{f-i,e-j} = \overline{\Sigma}_{i,j} = \bar{\omega}_{1,i,j}\cdots\bar{\omega}_{a,i,j}\omega_{a+1,i, j}\cdots\omega_{n,i,j}.
\]
Following the conventions established with $\omega_{i,j}$, we denote by $\Sigma_{i,j}$ the image of  $\Sigma_{i,j}$ under the standard comparison isomorphisms between Betti, deRham, and $\ell$-adic cohomology theories. 

The following proposition is due to Schreieder in the case when the curve $C$ is of the form $y^2 = x^{2g+1} + 1$ \cite[Lemma 8]{schreieder}.  The proof we give below is a straightforward generalization of his argument.

\begin{prop}\label{two pieces of cohomologygen}
\begin{enumerate}
\item The Hodge decomposition of the $\Q$-Hodge structure $H^*_B(M,\Q)=H^*_B(C^n, \Q)^G$ has the form
\[
H^*(C^n,\C)^G = V^{a, n-a} \oplus V^{n-a, a} \oplus \left(\bigoplus_{p=0}^n V^{p,p} \right).
\]
\item The subspace $V^{p,p}$ consists of all $G$-invariant homogeneous polynomials of degree $p$ in the classes $\Omega_1, \ldots, \Omega_n$.   
\item The set $\{\Sigma_{i,j} : 1 \leq i \leq f-1, 1 \leq j \leq e-1\}$ is a $\C$-basis for $V^{a, n-a} \oplus V^{n-a, a}$. 
\end{enumerate}
\end{prop}

\begin{proof}
Recall that for each $1 \leq k \leq n$, we have
$\psi_f^*(\omega_{k, i, j}) = \zeta_f^i\omega_{k, i, j}$ and  $\psi_e^*(\omega_{k,i,j}) = \zeta_e^{-j}\omega_{k, i, j}.$ 
Now, the cohomology ring $H^*(C^n,\C)$ is generated by the elements
$
\Omega_k, \omega_{k, i, j}, \overline{\omega}_{k, i, j}$ for $1 \leq k \leq n$, $ 1 \leq i \leq I,$ $1 \leq j \leq J$. Suppose there is a $G$-invariant class in $H^*(C^n,\C)$ that contains the monomial
\[
\eta = \Omega_{l_1} \cdots \Omega_{l_r} \omega_{k_1, i_1, j_1} \cdots \omega_{k_s,i_s,j_s} \overline{\omega}_{k_{s+1},i_{s+1},j_{s+1}} \cdots \overline{\omega}_{k_t,i_t,j_t},
\]
where, without loss of generality, we take $l_1\le \cdots \le l_r$,  $k_1\le\cdots \le k_s$, and $k_{s+1}\le \cdots \le k_t$.

Observe that the product of any $(1,0)$ and $(0,1)$ class of the factor $C_k$ lies in $H^{1,1}(C_k)$ and thus is a multiple of the fundamental class $\Omega_k$. Hence we may assume that the intersection
\[
\{k_\epsilon \mid 1 \leq \epsilon \leq s\} \cap \{k_\epsilon \mid s+1 \leq \epsilon \leq t\} = \emptyset.
\] 
Since $H^{2,0}(C_k)=0=H^{0,2}(C_k)$, we may assume that $k_1, \ldots, k_s$ are pairwise disjoint and the same for $k_{s+1}, \ldots, k_t$, so all of $k_1,\ldots, k_t$ are pairwise distinct. Moreover, since $H^{2,2}(C_k)=0$, the elements $l_1, \ldots, l_r$ are pairwise distinct.  Since $H^{2,1}(C_k)=0=H^{1,2}(C_k)$ we then have
\[
\{l_1, \ldots, l_r\} \cap \{k_\epsilon \mid 1 \leq \epsilon \leq t\} = \emptyset.
\]
That is, in the expression for $\eta$, all of $l_1, \ldots, l_r, k_1, \ldots, k_t$ are pairwise distinct. Namely we have $l_1< \cdots < l_r$,  $k_1<\cdots < k_s$, and $k_{s+1}< \cdots < k_t$.

Now note that the element $\psi_f^{u_1}\psi_e^{v_1} \times \cdots \times \psi_f^{u_n}\psi_e^{v_n}$ of the group $G$ acts on the monomial $\eta$ by multiplication by the scalar
\[
\zeta_f^{\mu_{u_1,\cdots, u_t}}\zeta_e^{\nu_{v_1,\cdots, v_t}},
\]
where $\mu_{u_1,\ldots,u_t}$ is of the form $i_1u_1 \pm \cdots \pm i_tu_t$ and $\nu_{v_1,\ldots,v_t}$ is of the form $j_1v_1\pm \cdots \pm j_tv_t$. 

First we show that $t<n$ implies that $t=0$. Indeed, begin by choosing values for $u_1,\ldots, u_t$ and $v_1,\ldots, v_t$ such that 
$\mu_{u_1,\ldots,u_t} \not\equiv 0 \bmod f$ and $\nu_{v_1,\ldots,v_t} \not\equiv 0 \bmod e$. Since $t<n$, we can still choose $u_{t+1}, \ldots, u_n$ and $v_{t+1}, \ldots, v_n$ such that 
\begin{align*}
u_1 + \cdots + u_a - u_{a+1} - \cdots - u_n &\equiv 0 \bmod f\\
v_1 + \cdots + v_a - v_{a+1} - \cdots - v_n &\equiv 0 \bmod e.
\end{align*}
Namely, the automorphism of $C_{F_{1,1}}^n$ given by $\psi_f^{u_1}\psi_e^{v_1} \times \cdots \times \psi_f^{u_n}\psi_e^{v_n}$ lies in the group $G$, but it acts non-trivially on the monomial $\eta$ when $t>0$. Since $\eta$ is $G$-invariant, we must have $t=0$ and $\eta = \Omega_{l_1} \cdots \Omega_{l_r}$.

Next suppose that $t=n$. By the pairwise distinctness of $l_1, \ldots, l_r, k_1, \ldots, k_t$, we know $r=0$.  We now show that in this case, we must have $i_1=\cdots=i_n$, $j_1=\cdots=j_n$, and either
\[
\{1, \ldots, a\} = \{k_1, \ldots, k_s\} \text{ or } \{1, \ldots, a\} = \{k_{s+1}, \ldots, k_n\}.
\]

Indeed, suppose $1\le \epsilon_1<\epsilon_2\le s$ are such that $k_{\epsilon_1}\le a$ and $k_{\epsilon_2}>a$. Then let $u_{\epsilon_1}=u_{\epsilon_2}=v_{\epsilon_1}=v_{\epsilon+2}=1$ and $u_{\epsilon}=v_{\epsilon}=0$ for all other $\epsilon \ne \epsilon_1,\epsilon_2$. Then the $n$-tuples $(u_1, \ldots, u_n)$  and $(v_1, \ldots, v_n)$ give rise to an element in $G$ with $\mu_{u_1,\ldots,u_n}=i_{\epsilon_1}+i_{\epsilon_2}$ and $\nu_{v_1,\ldots,v_n}=j_{\epsilon_1}+j_{\epsilon_2}$. But since the monomial $\eta$ must be $G$-invariant, it follows that
\begin{equation}\label{contradiction congruence f odd}
i_{\epsilon_1} + i_{\epsilon_2} \equiv 0 \bmod f
\end{equation}
\begin{equation} \label{contradiction congruence f even}
j_{\epsilon_1} + j_{\epsilon_2} \equiv 0 \bmod e.
\end{equation}
If $f$ is odd, then $1 \leq i_{\epsilon_1}, i_{\epsilon_2} \leq \frac{f-1}{2}$ and so \eqref{contradiction congruence f odd} is impossible.  If $f$ is even, then $1 \leq j_{\epsilon_1}, j_{\epsilon_2} \leq \frac{e-1}{2}$ and so \eqref{contradiction congruence f even} is impossible.  Namely, we have shown that there cannot exist $1\le \epsilon_1<\epsilon_2\le s$  with the property that $k_{\epsilon_1}\le a$ and $k_{\epsilon_2}>a$. Similarly, one checks that we cannot have $\epsilon_1, \epsilon_2 \in \{s+1, \ldots, n\}$ such that $k_{\epsilon_1} \leq a$ and $k_{\epsilon_2} > a$.  This proves that $\{1, \ldots, a\}$ is either $\{k_1, \ldots, k_s\}$ or $\{k_{s+1}, \ldots, k_n\}$.

It remains to show that $i_1 = \cdots = i_n$ and $j_1 =  \cdots = j_n$.  For any $\epsilon \leq a$, let $u_\epsilon = u_{a+1} = v_\epsilon=v_{a+1}=1$ and $u_\iota =v_\iota= 0$ for all $\iota \neq \epsilon, a+1$. As above, the $n$-tuples $(u_1, \ldots, u_n)$  and $(v_1, \ldots, v_n)$ give rise to an element in $G$ such that $i_\epsilon \equiv i_{a+1} \bmod f$ and $j_\epsilon \equiv j_{a+1} \bmod e$ since $\eta$ is $G$-invariant. But since $1\le i_\epsilon\le f-1$ and $1\le j_\epsilon\le e-1$, we have $i_\epsilon=i_{a+1}$ and $j_\epsilon=j_{a+1}$ for all $1\le \epsilon \le a$. Similarly, one checks that $i_a = i_\epsilon$ and $j_a = j_\epsilon$ for all $\epsilon > a$.  Therefore $i_1 = \cdots = i_n$ and $j_1 = \cdots = j_n$, as claimed.

Thus we have shown that the monomial $\eta$ must take one of the following forms:
\begin{align*}
&\Omega_{l_1} \cdots \Omega_{l_r}\\
&\Sigma_{i,j} \coloneqq \ \omega_{1, i, j} \cdots \omega_{a, i, j} \overline{\omega}_{a+1, i, j} \cdots \overline{\omega}_{n, i, j}\\
&\overline{\Sigma}_{i, j} \coloneqq \ \omega_{1, i, j} \cdots \omega_{n-a,i, j} \overline{\omega}_{n-a+1, i, j} \cdots \overline{\omega}_{n,i, j},
\end{align*}
with $1 \leq i \leq I, 1 \leq j \leq J$.  This proves the first two statements of the proposition.

Observe that all three types of forms are $G$-invariant. Moreover, for any $1\le i\le I$ and $1\le j\le J$, the form $\Sigma_{i,j}$ is of type $(a,n-a)$ and $\overline{\Sigma}_{i,j}$ is of type $(n-a,a)$ in the Hodge decomposition of $H^*_B(C^n,\C)^G$. Hence the $\C$-span of $\{\Sigma_{i,j} \mid 1 \leq i \leq I, 1 \leq j \leq J\}$ defines a $G$-invariant vector space $V^{a, n-a}$ of classes of type $(a, n-a)$, and the $\C$-span of $\{\overline{\Sigma}_{i, j} \mid 1 \leq i \leq I, 1 \leq j \leq J\}$ defines a $G$-invariant vector space $V^{n-a, a}$ of classes of type $(n-a, a)$.  The $\Sigma_{i,j}$ are linearly independent since they are tensor products of linearly independent elements.  Therefore $V^{a, n-a}$ and $V^{n-a, a}$ are $g$-dimensional and conjugate to each other by construction.
\end{proof}

\begin{cor}\label{betti motive basis}
The motive $\tilde{M}$ satisfies $H^*_B(\tilde{M}, \Q) \otimes \C = V^{a, n-a} \oplus V^{n-a, a}$.  Furthermore, a $\C$-basis for $H^*_B(M_{d,d'}, \Q) \otimes \C$ is
\[
\B_{d,d'} \coloneqq \{\Sigma_{i,j} \mid 1 \leq i \leq f-1, 1 \leq j \leq e-1, (i,f)=d, (j,e)=d'\}.
\]
\end{cor}

\begin{proof}
The first statement follows from the second, so we just prove the second statement.  Note that $H_B^*(M_{d,d'}, \Q)$ consists of the classes in $H_B^*(C^n, \Q)^G$ coming from $(e_{d,d'}H^1(C, \Q))^{\otimes n} = H_B^1(A_{d,d'}, \Q)^{\otimes n}$.  Since $B_{d,d'}$ is a basis for $H_B^1(A_{d,d'}, \Q) \otimes \C$ by Proposition \ref{decomposition of Jacobian gen}, it follows that $B_{d,d'}^{\otimes n} \cap (H_B^*(M_{d,d'}, \Q) \otimes \C)$ is a basis for $H_B^*(M_{d,d'}, \Q)\otimes \C$.  Proposition \ref{two pieces of cohomologygen} implies that $\B_{d,d'} = B_{d,d'}^{\otimes n} \cap (H_B^*(M_{d,d'}, \Q) \otimes \C)$, which completes the proof of the corollary.
\end{proof}

\begin{prop}\label{localLfactor2gen} Fix a rational prime $\ell$.  For pairs $(d,d')$ such that $d|f$, $d\ne f$ and $d'|e$, $d\ne e$ we have:
\begin{enumerate}
\item Let $\p \nmid \ell$ be a prime of $F$ where $A_{d,d'}$ has good reduction, and let $q = \#\OK_F/\p$.  The characteristic polynomial of $\Frob_\p^*$ acting on $H^*_{\et}(M_{d,d'},\Q_\ell)$ is 
\[
\prod_{c = 1}^{\frac{\varphi(f_d)}{\ord_{f_d} q}}\prod_{c' = 1}^{\frac{\varphi(e_{d'})}{\ord_{e_{d'}} q}} \left(T^{\ord_{f_d \cdot e_{d'}} q} - (Z_{(c,c')})^a(Z_{-(c,c')})^{n-a}\right),
\]
where the index of $Z_{(c,c')}$ is viewed as an element of $(\Z/f_d\Z \times \Z/e_{d'}\Z)^\times$.  In particular, $H_{\et}^*(M_{d,d'}, \Q_\ell)$ is a compatible system of $G_F$-representations.  
\item The local $L$-factor at $\p$ of $L(H^*_{\et}(M_{d,d'}, \Q_\ell)/F, s)$ is 
\[
\prod_{c = 1}^{\frac{\varphi(f_d)}{\ord_{f_d} q}}\prod_{c' = 1}^{\frac{\varphi(e_{d'})}{\ord_{e_{d'}} q}} \left(1 - (Z_{(c,c')})^a(Z_{-(c,c')})^{n-a}p^{-s\ord_{f_d \cdot e_{d'}} q}\right).
\]
\end{enumerate}
\end{prop}

\begin{proof}
We begin by computing the action of $\Frob_\p^*$ on the basis $\mathcal{B}_{d,d'}$. Recall that in Proposition \ref{Frob lem} we showed that the matrix $[\Frob_\p^*]_{B_{d,d'}}$ of $\Frob_\p^*$ acting on $H^1_{\et}(A_{d,d', \overline{F}}, \Q_\ell) \otimes_{\Q_\ell, \iota_\ell} \C$ with respect to the basis $B_{d,d'}$ is a generalized permutation matrix on $2g_{d,d'}=\varphi(f_d)\varphi(e_{d'})$ letters with corresponding permutation $\rho$ a product of $\frac{\varphi(f_d)\varphi(e_{d'})}{\ord_{f_d \cdot e_{d'}} q}$ disjoint cycles, each of length $\ord_{f_d \cdot e_{d'}} q$.  Recall the notation $I_{d, d'} = \{(i, j) \in \Z/f\Z \times \Z/e\Z \mid (i, f) = d, (j,e) = d'\}$.  For $s,t\in I_{d,d'}$, the generalized permutation matrix $[\Frob_\p^*]_{B_{d,d'}}$ is of the form $(a_{s,t})_{(s,t) \in I_{d,d'}}$ such that
\[
\Frob_\p^*(\omega_t) = a_{\rho^{-1}(t), t}\omega_{\rho^{-1}(t)}.
\]
Hence for $t \in I_{d,d'}$ we have
\begin{align*}
\Frob_\p^*(\Sigma_t) = &(a_{\rho^{-1}(t), t}\omega_{1, \rho^{-1}(t)}) \cdots (a_{\rho^{-1}(t), t}\omega_{a, \rho^{-1}(t)}) \\ 
&(a_{\rho^{-1}((e,f)-t), (e,f)-t}\omega_{a+1, \rho^{-1}((e,f)-t)}) \cdots (a_{\rho^{-1}((e,f)-t), (e,f)-t}\omega_{n, \rho^{-1}((e,f)-t)})\\
= &(a_{\rho^{-1}(t), t})^a(a_{\rho^{-1}((e,f)-t), (e,f)-t})^{n-a}\Sigma_{\rho^{-1}(t)}.
\end{align*}

The above calculation shows that the matrix $[\Frob_\p^*]_{\mathcal{B}_{d,d'}}$ of $\Frob_\p^*$ with respect to the basis $\mathcal{B}_{d,d'}$ of $H^*_{\et}(M_{d,d',\overline{F}}, \Q_\ell) \otimes_{\Q_\ell, \iota_\ell} \C$  is a generalized permutation matrix with associated permutation $\rho$. In particular, for $1 \leq c \leq \frac{\varphi(f_d)}{\ord_{f_d} q}$ and $1 \leq c' \leq \frac{\varphi(e_{d'})}{\ord_{e_{d'}} q}$, the product of all the nonzero entries in $[\Frob_\p^*]_{\mathcal{B}_{d,d'}}$ corresponding to the $(c,c')$-th disjoint cycle in $\rho$ is
\[
(Z_{(c,c')})^a(Z_{-(c,c')})^{n-a}.
\]
The result now follows from Lemma \ref{gpm}.  The second statement is a restatement of the first using the standard translation between characteristic polynomials and local $L$-factors.
\end{proof}

Recall that $S_{d,d'}$ is the set of primes of $F$ where $A_{d,d'}$ has bad reduction and $S = \cup_{d,d'} S_{d,d'}$.

\begin{cor}\label{Lfn equality}
Let $n\ge 1$ and $n/2<a\le n$, where we require $a=n$ if $F$ is not totally real. Then, we have an equality of (incomplete) $L$-functions:
\[
L^{(S_{d,d'})}(M_{d,d'}/F, s) = L^{(S_{d,d'})}(\lambda_{d,d'}^a\overline{\lambda}_{d,d'}^{n-a}, s).
\]
\end{cor}

\begin{proof}
This follows from Corollary \ref{equality of setsgen} and Proposition \ref{localLfactor2gen}.
\end{proof}

Therefore using the decomposition of $\tilde{M}$ given in \eqref{definition of M} and letting $S\coloneqq \cup_{d, d'} S_{d,d'}$, we have proven the following theorem.

\begin{thm}\label{big theorem} Let $n\ge 1$ and $n/2<a\le n$, where we require $a=n$ if $F$ is not totally real.  Then we have an equality of (incomplete) $L$-functions:
\[
L^{(S)}(\tilde{M}/F, s) = \prod_{d, d'}L^{(S)}(\lambda_{d,d'}^a\overline{\lambda}_{d,d'}^{n-a}, s).
\]
\end{thm}

\subsection{Relationship to other results in the literature}\label{literature}
We now briefly discuss how the motives constructed in Section \ref{group action} are related to other constructions of motives and varieties in the literature.  

\subsubsection{Standard motives of algebraic Hecke characters}\label{Schappacher comparison}
Given a number field $k$ and an algebraic Hecke character $\chi \colon \A_k^\times/k^\times \to \C^\times$, there is a standard way to construct a motive $M(\chi) \in \M_k^{\num}$ such that $L^{(S)}(M(\chi)/k, s) = L^{(S)}(\chi, s)$ for some finite set of places $S$ of $k$ \cite[\S I.4]{Schappacher}.  In fact, let $\M_k^{\num, \ab}$ be the Tannakian subcategory of $\mathcal{M}_k^{\num}$ generated by motives of abelian varieties and Artin motives over $k$.  Then $M(\chi)$ is characterized in the category $\M_k^{\num, \ab}$ by its $L$-function  \cite[Theorem I.5.1]{Schappacher}.  Note that the construction of the Chow motive $M_{d,d'}$ yields that the numerical realization of $M_{d,d'}$ lies in $\M_F^{\num, \ab}$.  Therefore by Corollary \ref{Lfn equality} and \cite[Theorem I.5.1]{Schappacher} we have 
\[
M_{d,d'} \times_F F_{d,d'} \cong M(\lambda_{d,d'}^a\overline{\lambda}_{d,d'}^{n-a}) \in \M_{F_{d,d'}}^{\num, \ab}.
\]
That is, our construction proves that the motives $M(\lambda_{d,d'}^a\overline{\lambda}_{d,d'}^{n-a}) \in \M_{F_{d,d'}}^{\num}$ descend to $F$ and can be realized at the level of Chow motives.

\subsubsection{Motives of CM modular forms}\label{Scholl comparison}
Suppose $F=\Q$, $K$ is an imaginary quadratic field, and $\chi$ is an algebraic Hecke character of $K$. Then one can form the theta series $\theta(\chi)$, which is a modular eigenform.  Scholl gave a construction attaching a homological (or \textit{Grothendieck}) motive over $\Q$ to any given eigenform \cite{Scholl1990}.  In the very limited circumstances when $F=\Q$ and $[K_{d,d'}:\Q] = 2$, our construction proves that Scholl's motives attached to $\theta(\lambda_{d,d'}^{a}\overline{\lambda}_{d,d'}^{n-a})$ can be realized at the level of Chow motives.

\subsubsection{Modularity of Schreieder's varieties}\label{Schreieder comparison}
The original motivation for this project was to prove that the varieties constructed by Schreieder in \cite{schreieder} are modular in a sense similar to \cite[Theorem 3.3]{cynk}.  Fix a positive integer $k\ge 1$ and let $e=2$, $f=3^k$, $\gamma=\delta=1$ in the Weil curve $C$ introduced in Section \ref{curves and jacobians}.  That is, $C$ has affine coordinate patches $y^2=x^{3^k}+1$, $v^2=u^{3^k+1}+u$ and genus $g=\frac{3^k-1}{2}$. The particular case of the group $G$ acting on $C^n$ was considered for $k=1$ and $a=n$ by Cynk-Hulek in \cite[Section 3]{cynk} and for general $k$ and $a$ by Schreieder in \cite[Section 8]{schreieder}.  In these papers they construct, for each $k\ge 1$, a smooth model $X$ of the singular quotient variety $C^n/G$ with the property:

\begin{prop}\cite[Theorem 17]{schreieder} For any $k\ge 1$ and $n/2 < a\le n$, the Betti cohomology of the $n$-dimensional variety $X$ is of the form
\[H_B^*(X,\C)=H_B^*(C^n,\C)^G\oplus \left( \bigoplus_{p=0}^n V^{p,p}\right),\]
where all classes in $V^{p,p}$ are algebraic. 
\end{prop}

In other words, the ``transcendental part" of $X$, which we denote $T(X)$, is equal to our motive $\tilde{M}$ in this case.  Write $A_i$ for the isotypic $3^{k-i-1}$-dimensional abelian variety defined over $\Q$ appearing in the decomposition of $(\Jac C)_{K_0}$ that obtains CM by $K_i = \Q(\zeta_{3^k}^{3^i})$ when base changed to $K_i$.  Let $S_i$ denote the finite set of primes of $\Q$ where $A_i$ has bad reduction, and let $\lambda_i$ be the Hecke character of $K_i$ associated to $A_i$ by Weil.  We obtain the following corollary.

\begin{cor}\label{geometric modularity}
For any choices of $k \ge 1$ and $n\ge 1$ and $n/2<a\le n$,  the $n$-dimensional smooth projective variety $X$ is modular.  That is, we have an equality of (incomplete) $L$-functions
\[
L^{(S)}(T(h(X))/\Q, s) = \prod_{i=1}^{k-1}L^{(S)}(\lambda_i^a\overline{\lambda}^{n-a}_i, s).
\]
\end{cor}

\begin{proof}
The particular decomposition of the Jacobian of $C$ is given by Corollary \ref{decomposition of Jacobian gen} and the result then follows from Theorem \ref{big theorem}.
\end{proof}

In particular, when $n$ is odd we obtain
\[
L^{(S)}(H^n_{\et}(X_{\overline{\Q}}), s) = \prod_{i = 1}^{k - 1}L^{(S)}(\lambda_i^a\overline{\lambda}_i^{n-a}, s).
\]
Corollary \ref{geometric modularity} was proved by Cynk and Hulek when $C$ is the elliptic curve $y^2 = x^3 + \delta$ ($\delta \in \Q$) and $a=n$ \cite[Theorem 3.3]{cynk}.

\textbf{Acknowledgements.}  The authors would like to thank the following people for helpful discussions in the preparation of this article: Don Blasius, Ashay Burungale, Francesc Castella, G\"unter Harder, Haruzo Hida, Matt Kerr, C\'edric P\'epin, Stefan Schreieder, Jacques Tilouine, Burt Totaro, Alberto Vezzani, Preston Wake, J\"org Wildeshaus.

The authors gratefully acknowledge support of the National Science Foundation through awards DGE-1144087 and DMS-1645877 (first author) and DMS-1604148 (second author).  The second author also gratefully acknowledges support from the Franco-American Fulbright Commission and the Max Planck Institute for Mathematics.

\normalsize{\bibliography{Modularity.bib}}
\bibliographystyle{alpha}

\end{document}